\documentclass[11pt]{amsart}
\usepackage{amsfonts}
\usepackage{amssymb,latexsym}
\usepackage{amsmath}
\usepackage{amsthm}
\usepackage{amssymb}
\usepackage{enumerate}
\newcommand{\R}{\mathbb R}

\newtheorem{theorem}{Theorem}[section]
\newtheorem{lemma}[theorem]{Lemma}

\newtheorem{corollary}[theorem]{Corollary}
\theoremstyle{definition}

\theoremstyle{remark}
\newtheorem{remark}[theorem]{Remark}
\numberwithin{equation}{section}

\begin{document}
\setlength{\baselineskip}{1.2\baselineskip}
\title[Curvature estimates for the Monge-Amp\`{e}re equation]{Curvature estimates for the level sets of solutions of the
 Monge-Amp\`{e}re equation $\det D^2 u=1$}
\author{Chuanqiang Chen}
\address{School of Mathematical Sciences\\
         University of Science and Technology of China\\
         Hefei, 230026, Anhui Province, CHINA}
\email{cqchen@mail.ustc.edu.cn}
\author{Xi-Nan Ma}
\address{School of Mathematical Sciences\\
         University of Science and Technology of China\\
         Hefei, 230026, Anhui Province, CHINA}
\email{xinan@ustc.edu.cn}
\author{Shujun Shi}
\address{School of Mathematical Sciences\\
         Harbin Normal University\\
         Harbin 150025, Heilongjiang Province, CHINA}
\email{shjshi@163.com}
\thanks{The second author was supported by the National Science Fund for Distinguished Young Scholars of
China and Wu Wen-Tsun Key Laboratory of Mathematics. The third author was partially supported by
the National Natural Science Foundation of China(Grant No. 11101110) and the Foundation of Harbin
Normal University.}
\maketitle

\begin{abstract}
For the Monge-Amp\`{e}re equation $\det D^2 u=1$, we find new auxiliary
curvature functions which attain respective maximum on the boundary. Moreover,
we obtain the upper bounded estimates for the Gauss curvature and mean
curvature of the level sets for the solution to this equation.
\end{abstract}

\section{Introduction}

Monge-Amp\`{e}re equation is one of the most important fully nonlinear partial differential equations.  It has the general form
$$\det D^2 u = f(x,u, Du).$$
Here $\det D^2u$ denotes the determinant of the Hessian matrix $D^2u$, $u$ is a function in the Euclicean space $\R^n$, and $f$ is a given function.
It is elliptic when the Hessian matrix $D^2u$ is positive definite, namely $u$ is strictly convex.  There is an extensive literature on the research
of the Monge-Amp\`{e}re equation, see Guti\'{e}rrez \cite{Gut}, Trudinger-Wang \cite{TW08} and the references therein.

In 2011, Hong-Huang-Wang \cite{HHW01} studied a class of degenerate elliptic Monge-Amp\`{e}re equation in a smooth, bounded and strictly convex domain
$\Omega$ of dimension $2$.
When they proved the existence of global smooth solutions to the homogeneous Dirichlet problem, they introduced the key auxiliary function $\mathcal{H}$,
 which is the product of curvature $\kappa$ of the level line of $u$ and the cubic of $|D u|$, and got the uniformly lower bound of $\mathcal{H}$
 on $\overline{\Omega}$. These imply an estimate for the lower bound of the curvature of the level line in some sense, which inspire us to study
 the following simplest homogeneous Dirichlet problem for elliptic Monge-Amp\`{e}re equation
\begin{equation}\label{61.1}
\left\{
\begin{array}{lcl}
 \det D^2 u = 1 &\rm{in}& \Omega,  \\[8pt]
 \qquad\;\;\; u = 0 &\rm{on}& \partial \Omega. \\
\end{array} \right.
\end{equation}
We find  appropriate functions called $P$ fucnctions, prove that the $P$
functions attain their maximum on the boundary and get the upper bounded
estimates for the Gauss curvature and mean curvature of the level sets.

In order to state our results, we need the standard curvature formula of the
level sets of a function (see Trudinger \cite{Tru97}). Firstly, we recall the definition of elementary symmetric functions.
For any $k=1,2, \cdots, n$, we set
$$ \sigma_k(\lambda)=\sum_{1\leq i_1<i_2<\cdots<i_k\leq n}\lambda_{i_1}\lambda_{i_2}\cdots\lambda_{i_k},
\quad \mbox{for any}\ \lambda=(\lambda_1,\lambda_2,\cdots, \lambda_n)\in \mathbb R^{n}.$$ Let $W=(w_{ij})$ be a symmetric $n\times n$ matrix,
 and we set $$\sigma_k(W)=\sigma_k(\lambda(W)),$$ where $\lambda(W)=(\lambda_1(W),\cdots, \lambda_n(W))$ are the eigenvalues of $W$.
 We also set $\sigma_0=1$ and $\sigma_k=0$ for any $k>n$.

Since the level sets of
the strictly convex solution in the problems \eqref{61.1} are convex with
respect to the normal direction $-D u$,  we have the following formula on the $m$-th curvature of the level sets of the solution $u$,
 $m=1,2,\cdots,n-1$,
$$
\sum\limits_{k,l = 1}^n {\frac{{\partial \sigma _{m+1} (D^2 u)}}{{\partial u_{kl} }}u_k u_l |D u|^{-m-2}}.
$$
When $m=1$, $$H=\sum\limits_{k,l = 1}^n {\frac{{\partial \sigma _{2} (D^2 u)}}{{\partial u_{kl} }}u_k u_l |D u|^{-3}}$$ is the mean
curvature of the level sets;
When $m=n-1$, $$K=\sum\limits_{k,l = 1}^n {\frac{{\partial \sigma _{n} (D^2 u)}}{{\partial u_{kl} }}u_k u_l |D u|^{-n-1}}$$ is the
Gauss curvature of the level sets.

\begin{theorem}\label{t1.1}
Let $\Omega \subset \R^n$ be a bounded convex domain, $n\geq 2$, $u$ the strictly convex solution of \eqref{61.1}. Then the function
\[ \varphi=\sum\limits_{k,l = 1}^n {\frac{{\partial \sigma _n (D^2 u)}}{{\partial u_{kl} }}u_k u_l }-2u
\]
attains its maximum on the boundary $\partial \Omega$.
\end{theorem}

\begin{theorem}\label{t1.2}
Under the same assumptions as in the above theorem, we have the function
\[
\psi=\sum\limits_{k,l = 1}^n {\frac{{\partial \sigma _2 (D^2 u)}}{{\partial u_{kl} }}u_k u_l }-2(n-1)u
\]
also attains its maximum  on the boundary $\partial \Omega$. Moreover, $\psi$ attains its maximum in $\Omega$ if and only if $\Omega$ is
an ellipse for $n=2$ or a ball for $n\geq 3$.
\end{theorem}

Naturally, we have the following corollary.
\begin{corollary}\label{c1}
Let $\Omega$ be a smooth, bounded and strictly convex domain in $\R^n$, $n\geq 2$. If $u$ is the solution of the problem \eqref{61.1}, then the functions $K|D u|^{n+1}$ and $H|D u|^{3}$ attain their maximum only at the boundary ~$\partial \Omega$. Thus, for $ x\in \Omega\setminus\Omega'$, we have the following estimates
$$ K(x)< \frac{\max_{\partial\Omega}K \max_{\partial\Omega'}\kappa_M^{n+1}}{\min_{\partial\Omega}\kappa_m^{n+1}}$$
and
$$H(x)< \frac{\max_{\partial\Omega}H \max_{\partial\Omega'}\kappa_M^{n+1}}{\min_{\partial\Omega}\kappa_m^{n+1}},$$
where $\Omega'=\{x\in \Omega \ |u(x)<c, c\in(\min_{\Omega}u, 0)\; \text{is a constant}\}$, $\kappa_m, \kappa_M$ are  minimal and maximal principal curvatures of the level sets at a point respectively.
\end{corollary}

It should be mentioned that for the case $n=2$, Ma \cite{Ma98} and Anedda-Porru \cite{AP06} considered the problem \eqref{61.1}
and arrived at the conclusion of Theorem \ref{t1.1}. When $n=2$, there is only one curvature $\kappa$ for  the level sets at a point,
so $\kappa = K=H$ and ~$\varphi=\psi$ in Theorem \ref{t1.1} and Theorem \ref{t1.2}.
And $\mathcal{H}=\kappa|D u|^3=\sum\limits_{k,l = 1}^2 \frac{{\partial \sigma _{2} (D^2 u)}}{{\partial u_{kl} }}u_k u_l$ is the
auxiliary function introduced by Hong-Huang-Wang \cite{HHW01}.

There are also many papers to study curvature estimates for the level sets of
solutions of partial differential equation, see \cite{Lo83,CMY10, ChSh10, GX10,
JMO08To, MOZ10} etc..

This paper is organized as follows. In Section 2, we prove Theorem \ref{t1.1} through establishing a differential inequality for the
given function. In Section 3, using the same process as
 the proof of Theorem \ref{t1.1}, we prove the first results in Theorem \ref{t1.2}. Through the computation of the third derivatives
 for the solution $u$, we prove the relation between $\psi$ attaining its maximum in the interior and the shape of the domain $\Omega$.
 Finally, we prove the corollary and give some remarks.

\section{Proof of Theorem \ref{t1.1}}

\begin{proof} Let $D^2u=(u_{ij}), (u^{ij})=(u_{ij})^{-1}.$ Because $u$ is the convex solution of the equation
$\sigma _n (D^2 u)=\det{(D^2 u)}=1$, $(u_{ij})$ is positive definite and
$\frac{{\partial \sigma _n (D^2 u)}}{{\partial u_{kl} }}=u^{kl}.$ Therefore,
\begin{equation}
\varphi = \sum\limits_{k,l = 1}^n {\frac{{\partial \sigma _n (D^2 u)}}{{\partial u_{kl} }}u_k u_l-2u=\sum\limits_{k,l = 1}^n u^{kl}u_k u_l}-2u.
\end{equation}
We will prove the following differential inequality:

\begin{equation}\label{2.3}
\sum\limits_{i,j= 1}^n u^{ij} \varphi _{ij} \geq 0 \quad \text{ in } \Omega.
\end{equation}
From the differential inequality, by the maximum principle, $\varphi$ attains its maximum on the boundary $\partial \Omega$.

In the following, we will prove \eqref{2.3}. For any $x_o \in \Omega$, we choose coordinates such that
$(u_{ij}(x_o))$ is diagonal. All the following calculations are done at $x_o$.

Let ${u^{ij}}_{k}=\frac{\partial u^{ij}}{\partial x_k}, {u^{ij}}_{kk}=\frac{\partial^2 u^{ij}}{\partial x_k^2}$. From the direct computations, we have
\begin{eqnarray}\label{2.4}
\varphi_i&=&(\sum\limits_{k,l = 1}^n u^{kl}u_ku_l)_i-2u_i\nonumber\\
&=&\sum\limits_{k,l = 1}^n({u^{kl}}_iu_ku_l+2u^{kl}u_{ki}u_l)-2u_i\nonumber\\
&=&\sum\limits_{k,l = 1}^n{u^{kl}}_iu_ku_l
\end{eqnarray}
and
\begin{eqnarray}\label{2.5}
\varphi_{ii}&=&\sum\limits_{k,l = 1}^n({u^{kl}}_{ii}u_ku_l+2{u^{kl}}_iu_{ki}u_l)\nonumber\\
&=&\sum\limits_{k,l = 1}^n{u^{kl}}_{ii}u_ku_l+2\sum\limits_{l = 1}^n{u^{il}}_iu_{ii}u_l.
\end{eqnarray}
Thus
\begin{eqnarray}\label{2.6}
&&\sum\limits_{i,j = 1}^nu^{ij}\varphi_{ij}=\sum\limits_{i= 1}^nu^{ii}\varphi_{ii}\nonumber\\
&=&\sum\limits_{i,k,l = 1}^nu^{ii}{u^{kl}}_{ii}u_ku_l+2\sum\limits_{i,l = 1}^n{u^{il}}_iu_l\nonumber\\
&=&\sum\limits_{i,k,l = 1}^nu^{ii}{u^{kl}}_{ii}u_ku_l,
\end{eqnarray}
where we have used $\sum\limits_{i = 1}^n{u^{il}}_i=0$ in the last equality above.
Since
\begin{eqnarray}\label{2.61}
&&{u^{kl}}_{ii}=(-\sum\limits_{p,q = 1}^nu^{kq}u^{pl}u_{pqi})_i\nonumber\\
&=&\sum\limits_{p,q,r,s= 1}^n(u^{ks}u^{rq}u^{pl}+u^{kq}u^{ps}u^{rl})u_{pqi}u_{rsi}-\sum\limits_{p,q = 1}^nu^{kq}u^{pl}u_{pqii}\nonumber\\
&=&2\sum\limits_{j= 1}^nu^{kk}u^{ll}u^{jj}u_{jki}u_{jli}-u^{kk}u^{ll}u_{klii},
\end{eqnarray}
substituting ~\eqref{2.61} into ~\eqref{2.6}, we obtain
\begin{eqnarray}\label{2.7}
&&\sum\limits_{i,j = 1}^nu^{ij}\varphi_{ij}=\sum\limits_{i= 1}^nu^{ii}\varphi_{ii}\nonumber\\
&=&2\sum\limits_{i,j,k,l = 1}^nu^{ii}u^{jj}u^{kk}u^{ll}u_{ijk}u_{ijl}u_ku_l-\sum\limits_{i,k,l = 1}^nu^{kk}u^{ll}u^{ii}u_{iikl}u_ku_l.
\end{eqnarray}
Because of the equation ~$\det(u_{ij})=1$, differentiating it once, we can get
\begin{equation*}
\sum\limits_{i,j= 1}^nu^{ij} u_{ijk} = 0,
\end{equation*}
i.e.
\begin{equation}\label{2.8}
\sum\limits_{i= 1}^nu^{ii} u_{iik} = 0.
\end{equation}
Differentiating the equation once again, we have
$$
\sum\limits_{i,j= 1}^n u^{ij}u_{ijkl} + \sum\limits_{i,j,p,q = 1}^n(-u^{iq}u^{pj} u_{ijk} u_{pql}) = 0,
$$
i.e.
\begin{equation} \label{2.9}
\sum\limits_{i= 1}^n u^{ii}u_{iikl}= \sum\limits_{i,j = 1}^nu^{ii}u^{jj} u_{ijk} u_{ijl}.
\end{equation}
Substituting \eqref{2.9} into \eqref{2.7}, we obtain
\begin{eqnarray}\label{2.10}
&&\sum\limits_{i,j = 1}^nu^{ij}\varphi_{ij}=\sum\limits_{i= 1}^nu^{ii}\varphi_{ii}\nonumber\\
&=&2\sum\limits_{i,j,k,l = 1}^nu^{ii}u^{jj}u^{kk}u^{ll}u_{ijk}u_{ijl}u_ku_l-\sum\limits_{i,j,k,l = 1}^nu^{kk}u^{ll}u^{ii}u^{jj} u_{ijk} u_{ijl}u_ku_l\nonumber\\
&=&\sum\limits_{i,j,k,l = 1}^nu^{ii}u^{jj}u^{kk}u^{ll}u_{ijk}u_{ijl}u_ku_l\nonumber\\
&=&\sum\limits_{i,j= 1}^nu^{ii}u^{jj}(\sum\limits_{k= 1}^nu_{ijk}u^{kk}u_k)^2\nonumber\\
&\geq& 0.
\end{eqnarray}
We have completed the proof of Theorem \ref{t1.1}. \end{proof}

\begin{remark}\label{r1}

When $n=2$ , Ma \cite{Ma98} and  Anedda-Porru \cite{AP06} gave the result of Theorem \ref{t1.1} and furthermore pointed out that~$\varphi$ assumes its minimum on~$\partial \Omega$ or at the unique critical point $x_0$ of $u$, i.e. the point where ~$D u=0$. We can also get the conclusion from the above proof directly. In fact, from~\eqref{2.4}, we get
$$-\varphi_i=\sum_{k,l=1}^2 u^{kk}{u^{ll}}u_{kli}u_ku_l,$$
that is
\begin{equation*}
\left \{
\begin{array}{l}
(u^{11}u_1)^2u_{111}+2u_1u_2u_{112}+(u^{22}u_2)^2u_{122}=-\varphi_1\\[8pt]
(u^{11}u_1)^2u_{112}+2u_1u_2u_{122}+(u^{22}u_2)^2u_{222}=-\varphi_2,
\end{array}\right.
\end{equation*}
Here we have used the equation~$\det(u_{ij})=1$. Combining with \eqref{2.8}, that is
\begin{equation*}
\left \{
\begin{array}{l}
u^{11}u_{111}+u^{22}u_{122}=0\\[8pt]
u^{11}u_{112}+u^{22}u_{222}=0,
\end{array}\right.
\end{equation*}
Under the case of modulo $D \varphi$, we obtain the homogeneous linear algebraic system about the third derivatives $u_{111}, u_{112}, u_{122}, u_{222}$ of $u$,
\begin{equation}\label{2.11}
\left \{
\begin{array}{l}
u^{11}u_{111}+u^{22}u_{122}=0\\[8pt]
u^{11}u_{112}+u^{22}u_{222}=0\\[8pt]
(u^{11}u_1)^2u_{111}+2u_1u_2u_{112}+(u^{22}u_2)^2u_{122}=0\\[8pt]
(u^{11}u_1)^2u_{112}+2u_1u_2u_{122}+(u^{22}u_2)^2u_{222}=0.
\end{array}\right.
\end{equation}

From direct computations, we get that the determinant of the coefficient matrix is $(u^{11}u_1^2+u^{22}u_2^2)^2$, which is greater than 0 in $\Omega\setminus\{x_0\}$.
Therefore $u_{111}=u_{112}=u_{122}=u_{222}=0\quad {\rm{mod}}\ (D \varphi)$.
Consequently, from \eqref{2.10}, we have
$$ \sum\limits_{i,j = 1}^2u^{ij}\varphi_{ij}=0\quad {\rm{mod}}\ (D \varphi)\ \ {\rm{in}}\ \Omega\setminus \{x_0\}.$$
and by maximum principle, $\varphi$ attains its minimum on ~$\partial \Omega$ or at the unique critical points $x_0$.
\end{remark}

\section{Proof of Theorem \ref{t1.2}}

\begin{proof} Let ${\frac{{\partial \sigma _2 (D^2 u)}}{{\partial u_{kl} }}}= b^{kl}$. Then
\begin{equation} \label{3.1}
\psi  = \sum\limits_{k,l = 1}^n b^{kl}u_k u_l -2(n-1)u.
\end{equation}
We will prove the following differential inequality:
\begin{equation}\label{3.2}
\sum\limits_{i,j= 1}^n {u^{ij} \psi _{ij } } \geq 0 \quad \text{ in } \Omega.
\end{equation}
From the differential inequality, by the maximum principle, $\psi$ attains its maximum on the boundary $\partial \Omega$.

In the following, we will prove the differential equality \eqref{3.2}. For any $x_o \in \Omega$, we choose coordinates such that
$D^2u(x_o)$ is diagonal. All the following calculations are done at $x_o$.

Let ${b^{ij}}_{k}=\frac{\partial b^{ij}}{\partial x_k}, {b^{ij}}_{kk}=\frac{\partial^2 b^{ij}}{\partial x_k^2}$. From the direct computations, we have
\begin{eqnarray}
\psi_i&=&(\sum\limits_{k,l = 1}^n b^{kl}u_ku_l)_i-2(n-1)u_i\nonumber\\
&=&\sum\limits_{k,l = 1}^n({b^{kl}}_iu_ku_l+2b^{kl}u_{ki}u_l)-2(n-1)u_i\nonumber\\
&=&\sum\limits_{k,l = 1}^n{b^{kl}}_iu_ku_l+2b^{ii}u_{ii}u_i-2(n-1)u_i
\end{eqnarray}
and
\begin{eqnarray}
\psi_{ii}&=&\sum\limits_{k,l = 1}^n({b^{kl}}_{ii}u_ku_l+4{b^{kl}}_iu_{ki}u_l+2b^{kl}u_{kii}u_l+2b^{kl}u_{ki}u_{li})\nonumber\\
&&-2(n-1)u_{ii}\nonumber\\
&=&\sum\limits_{k,l = 1}^n{b^{kl}}_{ii}u_ku_l+4\sum\limits_{l = 1}^n{b^{il}}_iu_{ii}u_l+2\sum\limits_{k = 1}^nb^{kk}u_{kii}u_k\nonumber\\
&&+2b^{ii}u_{ii}^2-2(n-1)u_{ii}.
\end{eqnarray}
Therefore,
\begin{eqnarray}\label{3.3}
&&\sum\limits_{i,j = 1}^nu^{ij}\psi_{ij}=\sum\limits_{i= 1}^nu^{ii}\psi_{ii}\nonumber\\
&=&\sum\limits_{i,k,l = 1}^nu^{ii}{b^{kl}}_{ii}u_ku_l+4\sum\limits_{i,l = 1}^n{b^{il}}_iu_l+2\sum\limits_{i,k = 1}^nb^{kk}u^{ii}u_{iik}u_k\nonumber\\
&&+2\sum\limits_{i = 1}^nb^{ii}u_{ii}-2n(n-1)\nonumber\\
&=&\sum\limits_{i,k,l = 1}^nu^{ii}{b^{kl}}_{ii}u_ku_l+4\sigma_2(D^2u)-2n(n-1),
\end{eqnarray}
where we have used ~\eqref{2.8} and
$$\sum\limits_{i = 1}^n{b^{il}}_i=0\quad\text{and}\quad \sum\limits_{i = 1}^nb^{ii}u_{ii}=2\sigma_2(D^2u)$$ in the last equality above.
Since
\begin{equation*}
{b^{kl}}=
\begin{cases}
\sum\limits_{j= 1,j\neq k}^n u_{jj},&\quad k=l,\\
\qquad\ -u_{kl},&\quad k\neq l,
\end{cases}
\end{equation*}
we have
\begin{equation}\label{3.31}
{b^{kl}}_{ii}=
\begin{cases}
\sum\limits_{j= 1,j\neq k}^n u_{jjii},&\quad k=l,\\
\qquad\ -u_{klii},&\quad k\neq l.
\end{cases}
\end{equation}
Substituting \eqref{3.31} into \eqref{3.3}, we can get
\begin{eqnarray}\label{3.4}
& &\sum\limits_{i,j = 1}^nu^{ij}\psi_{ij}=\sum\limits_{i= 1}^nu^{ii}\psi_{ii}\nonumber\\
&=&\sum\limits_{\substack{i,j,k= 1\\j\neq k}}^nu^{ii}u_{iijj}u_k^2-\sum\limits_{\substack{i,k,l = 1\\k\neq l}}^n u^{ii}u_{iikl}u_ku_l\nonumber\\
&&+(4\sigma_2(D^2u)-2n(n-1))\nonumber\\
&=& \sum\limits_{\substack{i,j,k,l=1\\ k\neq l}}^n (u^{ii}u^{jj}u_{ijl}^2u_k^2-u^{ii}u^{jj}u_{ijk}u_{ijl}u_ku_l)\nonumber\\
&&+(4\sigma_2(D^2u)-2n(n-1)),
\end{eqnarray}
where we have used \eqref{2.9} in the last equality above.
We also have
\begin{equation}\label{3.5}
\sigma_2(D^2u) \geq C_n^2 (\sigma_n(D^2u))^{\frac{2}{n}} = C_n^2 = \frac{n(n-1)}{2}
\end{equation}
by Newton's inequality (see Hardy-Littlewood-P\'{o}lya \cite[section 2.22]{HLP}) and
\begin{equation}\label{3.6}
\sum\limits_{k,l=1}^n(u_{ijl}^2u_k^2-u_{ijk}u_{ijl}u_ku_l)\geq 0
\end{equation}
by Cauchy-Schwarz's inequality. Combining \eqref{3.4}, \eqref{3.5} and \eqref{3.6}, We can obtain that
\begin{eqnarray}\label{3.7}
&&\sum\limits_{i,j = 1}^n u^{ij}\psi_{ij}=\sum\limits_{i= 1}^nu^{ii}\psi_{ii}\nonumber\\
&=&\sum\limits_{\substack{i,j,k,l=1\\k\neq l}}^n (u^{ii}u^{jj}u_{ijl}^2u_k^2-u^{ii}u^{jj}u_{ijk}u_{ijl}u_ku_l)\nonumber\\
&&+(4\sigma_2(D^2u)-2n(n-1))\nonumber\\
&=&\sum\limits_{i,j=1}^n (u^{ii}u^{jj}\sum\limits_{k,l=1}^n(u_{ijl}^2u_k^2-u_{ijk}u_{ijl}u_ku_l))\nonumber\\
&&+(4\sigma_2(D^2u)-2n(n-1))\nonumber\\
&\geq& 0.
\end{eqnarray}

Furthermore, when $n=2$, if $\psi$ attains its maximum in $\Omega$, then $\psi=\varphi$ and  $\Omega$ is an ellipse by Remark \ref{r1} and vice versa. We also can get that $\Omega$ is an ellipse from \eqref{3.7}. In fact, if $\psi$ attains its maximum in $\Omega$, then ~$\psi$ is a constant in $\Omega$. So~\eqref{3.7} is the equality, that is
\begin{eqnarray*}
0&=&\sum\limits_{i,j=1}^2 (u^{ii}u^{jj}\sum\limits_{k,l=1}^n(u_{ijl}^2u_k^2-u_{ijk}u_{ijl}u_ku_l))\nonumber\\
&=&(u^{11})^2(u_{111}u_2-u_{112}u_1)^2+(u^{22})^2(u_{122}u_2-u_{222}u_1)^2+2(u_{112}u_2-u_{122}u_1)^2.
\end{eqnarray*}
Because $u^{11}>0, u^{22}>0$,
\begin{equation}\label{3.8}
\left\{
\begin{array}{l}
u_{111}u_2-u_{112}u_1=0\\[8pt]
u_{122}u_2-u_{222}u_1=0\\[8pt]
u_{112}u_2-u_{122}u_1=0.
\end{array}\right.
\end{equation}
Since~$u$ is strictly convex, it has the unique critical point $x_0$,  $(u_1,u_2)\neq (0,0)$ in $\Omega\setminus\{x_0\}$. From the theory of linear algebraic systems, we have that the rank of the coefficient matrix of the system \eqref{3.8} about $u_1, u_2$ is less than 2 and so
\begin{equation}\label{3.9}
\left\{
\begin{array}{l}
u_{112}^2-u_{111}u_{122}=0\\[8pt]
u_{122}^2-u_{112}u_{222}=0\\[8pt]
u_{111}u_{222}-u_{112}u_{122}=0.
\end{array}\right.
\end{equation}
By \eqref{2.8}, we get
\begin{equation}\label{3.10}
\left\{
\begin{array}{l}
u_{111}u^{11}+u_{122}u^{22}=0\\[8pt]
u_{112}u^{11}+u_{222}u^{22}=0.
\end{array}\right.
\end{equation}
Combining \eqref{3.9} with \eqref{3.10}, we obtain
$$u_{111}=u_{112}=u_{122}=u_{222}=0\;  \text{in}\;\Omega\setminus\{x_0\}.$$
and so all the third derivatives of $u$ vanish in~$\Omega$ by the continuousness. Consequently,
$\Omega=\{u< 0\}$ must be an ellipse.

When $n\geq 3$,  if $\psi$ attains its maximum in $\Omega$, then $\psi$ is a constant in $\Omega$. So \eqref{3.7} is the equality and we must have $4\sigma_2(D^2u)-2n(n-1)=0$, that is, the equality holds in \eqref{3.5}. But the equality holds in the Newton's inequality if and only if all the eigenvalues of $D^2u$ are equal. Therefore, the eigenvalues of $D^2u$ are equal to $1$ by the equation $\det D^2 u=1$ and $D^2 u$ is the unit matrix. Consequently,  $$u=\frac12 (|x-x_0|^2- r^2),$$
where $x_0\in \R^n$ is a fixed point and $r>0$ is a constant, and $\Omega=\{u< 0\}=B_r(x_0)$ is a ball. On the other hand, if $\Omega=B_r(x_0)$ is a ball, then the solution for the problem \eqref{61.1} is $u=\frac12 (|x-x_0|^2- r^2)$ and $\psi \equiv (n-1)r^2$ is a constant.

We have completed the proof of the Theorem \ref{t1.2}.\end{proof}

\section{Proof of the Corollary \ref{c1}}

We firstly give the boundary estimate of the gradient $D u$ for the solution of \eqref{61.1}.
\begin{lemma}\label{l1}

Let $\Omega $ be a smooth, bounded and strictly convex domain in $\R^n$, $x\in \partial \Omega$ and $\kappa_i(x),i=1,2,\cdots, n-1,$ be the principal curvatures of $\partial \Omega$ at $x$.
Let
$$\kappa_m(x)=\min\{\kappa_i(x)| i=1,2,\cdots,n-1\},
\kappa_M(x)=\max\{\kappa_i(x)| i=1,2,\cdots,n-1\}.$$
If $u$ is the smooth and strictly convex solution of \eqref{61.1}, then  on the boundary $\partial \Omega$, $|D u|_{\partial \Omega}$ satisfies the following estimate:
\begin{equation}\label{6L1}
\frac{1}{\max_{\partial \Omega}\kappa_{M}}\leq |D u|_{\partial \Omega}\leq \frac{1}{\min_{\partial \Omega}\kappa_{m}}.
\end{equation}
The same estimate is true for $\Omega'=\{x\in \Omega \ |u(x)<c, c\in(\min_{\Omega} u, 0)\; \text{is a constant}\}$, that is,
on the boundary $\partial \Omega'$, $|D u|_{\partial \Omega'}$ satisfies
\begin{equation}\label{6L2}
\frac{1}{\max_{\partial \Omega'}\kappa_{M}}\leq |D u|_{\partial \Omega'}\leq \frac{1}{\min_{\partial \Omega'}\kappa_{m}}.
\end{equation}
\end{lemma}

\begin{proof}
For any boundary point $x$, let $\Omega\subseteq \Omega_0 $ and $\Omega_1\subseteq \Omega$ be two
balls of radius~$R=\frac{1}{\min_{\partial \Omega}\kappa_{m}}$ and $r=\frac1{\max_{\partial \Omega}\kappa_{M}}$ respectively and $x\in \bar{\Omega} \cap \bar{\Omega}_j, j = 0, 1$.
Let $u_{\Omega_j},j=0,1,$  be the solution to the problem
\begin{equation*}
\left\{
\begin{array}{lcl}
\det D^2u =1 &\mbox{in}& \Omega_j,\\[8pt]
\qquad\;\;\; u=0 &\mbox{on}& \partial \Omega_j.
\end{array} \right.
\end{equation*}

Since $u$ vanishes on $\partial \Omega$, it follows immediately that
$$|D u_{\Omega_1}(x)|\leq|D u(x)|\leq |D u_{\Omega_0}(x)|.$$
An explicit calculation yields
$$|D u_{\Omega_1}(x)|= r,\quad |D u_{\Omega_0}(x)|= R,$$
and so $$r\leq |D u(x)|\leq R.$$
Therefore,
$$\frac{1}{\max_{\partial \Omega}\kappa_{M}}\leq |D u|_{\partial \Omega}\leq \frac{1}{\min_{\partial \Omega}\kappa_{m}},$$
and \eqref{6L1} holds. For the same reasons, \eqref{6L2} also holds.
\end{proof}

Next, we start the proof of Corollary \ref{c1}.
\begin{proof}

By Theorem \ref{t1.1}, we have $K|D u|^{n+1}-2u$ takes its maximum on the boundary $\partial\Omega$. For any $x\in \Omega,$ we have
$$K(x)|D u(x)|^{n+1}-2u(x)\leq \max_{\partial\Omega}{K|D u|^{n+1}},$$
and so, by $u(x)<0$,
$$K(x)|D u(x)|^{n+1}\leq \max_{\partial\Omega}{K|D u|^{n+1}}+2u(x)< \max_{\partial\Omega}{K|D u|}.$$
Therefore $K|D u|^{n+1}$ attains its maximum only at the boundary $\partial \Omega$. For the same reasons, by Theorem \ref{t1.2}, we get
~$H|D u|^{3}$ also attains its maximum only at the boundary $\partial \Omega$.
Since $u$ is strictly convex,  $|D u|$ increases along the increasing direction of the level sets. By Lemma \ref{l1}, we have, for $x\in \Omega\setminus\Omega'$,
$$K(x)<\frac{\max_{\partial\Omega}(K|D u|^{n+1})}{|D u(x)|^{n+1}}\leq \frac{\max_{\partial\Omega}(K|D u|^{n+1})}{\min_{\partial\Omega'}|D u|^{n+1}}\leq \frac{\max_{\partial\Omega}K \max_{\partial\Omega'}\kappa_M^{n+1}}{\min_{\partial\Omega}\kappa_m^{n+1}}.$$
For the same reasons,
$$H(x)< \frac{\max_{\partial\Omega}H \max_{\partial\Omega'}\kappa_M^{n+1}}{\min_{\partial\Omega}\kappa_m^{n+1}}.$$
\end{proof}

\begin{remark}\label{r2}
When $n=2$, by Remark {\rm\ref{r1}},  $\varphi$ attains its minimum on~$\partial \Omega$ or at the unique critical point $x_0$ of $u$. Therefore, we can furthermore give the positive lower bound estimate for the curvature of the level lines. In fact,
for any $ x\in \Omega\setminus\Omega'$, we have
\begin{align*}
\kappa(x)|D u(x)|^{3}&\geq \min\{\min_{\partial \Omega}\kappa\min_{\partial \Omega}|D u|^{3}, -2u(x_0)\}+2u(x)\\
&\geq \min\{\min_{\partial \Omega}\kappa\min_{\partial \Omega}|D u|^{3}+2c, 2c-2u(x_0)\}.
\end{align*}
Since $$\min_{\partial \Omega}|D u|^{3}\geq\frac1{\max_{\partial \Omega}\kappa^{3}},$$
$$|D u(x)|^{3}\leq \max_{{\partial \Omega}}|D u|^{3}\leq \frac1{\min_{{\partial \Omega}}\kappa^{3}},$$
we obtain
$$\kappa(x)\geq (\min_{{\partial \Omega}}\kappa^{3})(\min\{\frac{\min_{\partial \Omega}\kappa} {\max_{\partial \Omega}\kappa^{3}}+2c, 2c-2u(x_0)\}).$$

\end{remark}

\begin{remark}
It is more interesting to obtain the lower bound estimate for the curvature of
 the level sets for Monge-Amp\`{e}re
equations in higher dimensions. If it is true then it may be helpful to improve
the regularity for solutions of degenerate elliptic Monge-Amp\`{e}re equations
in higher dimensions as in {\rm Hong-Huang-Wang \cite{HHW01}}.
\end{remark}

\end{document}